\documentclass[11pt]{article}
\usepackage{amsthm, amsmath, amssymb, amsfonts, url, booktabs, tikz, setspace, fancyhdr, enumerate}
\usepackage[margin = 1in]{geometry}

\usepackage[czech,english]{babel} 

\usepackage[hidelinks]{hyperref}
\usepackage{tikz}
\usepackage{comment}
\usepackage{xcolor}
\usepackage{cite}

\usepackage[textsize=scriptsize,colorinlistoftodos]{todonotes}

\newtheorem{theorem}{Theorem}[section]

\newtheorem{lemma}[theorem]{Lemma}
\newtheorem{corollary}[theorem]{Corollary}
\newtheorem{conjecture}[theorem]{Conjecture}

\theoremstyle{definition}
\newtheorem{definition}[theorem]{Definition}

\theoremstyle{remark}

\newcommand{\x}{\times}

\newcommand{\Hom}{\mathrm{Hom}}

\DeclareMathOperator*{\E}{\mathbb E}
\newcommand{\PP}{\mathbb{P}}
\newcommand{\RR}{\mathbb{R}}

\newcommand{\HH}{\mathbb{H}}

\newcommand{\ev}{\mathcal{E}}
\newcommand{\FF}{\mathcal{F}}
\renewcommand{\S}{\mathcal{S}}

\newcommand{\TT}{\mathcal{T}}
\newcommand{\X}{\mathbf{X}}
\newcommand{\Y}{\mathbf{Y}}

\tikzset{vtx/.style={inner sep=1.7pt, outer sep=0pt, circle, fill,draw}}

\title{On tripartite common graphs}
\author{
Andrzej Grzesik\thanks{Faculty of Mathematics and Computer Science, Jagiellonian University, {\L}ojasiewicza 6, 30-348 Krak\'{o}w, Poland, E-mail: {\tt Andrzej.Grzesik@uj.edu.pl}. Research supported in part by ERC Consolidator Grant LaDIST 648509.}\and
Joonkyung Lee\thanks{
Department of Mathematics, Hanyang University, 222 Wangsimni-ro, Seongdong-gu, Seoul, South Korea.
E-mail: {\tt
joonkyunglee@hanyang.ac.kr}. Research supported in part by ERC Consolidator Grant PEPCo 724903.}\and
Bernard Lidick\'{y}\thanks{Department of Mathematics, Iowa State University. Ames, IA, USA. E-mail: \texttt{lidicky@iastate.edu} Supported in part by NSF grant DMS-1855653.}\and
Jan Volec\thanks{
Department of Mathematics, Faculty of Nuclear Sciences and Physical Engineering, Czech Technical University in Prague, Trojanova 13, 120 00 Prague, Czech Republic.
E-mail: \texttt{jan@ucw.cz}
Previous affiliation: Department of Mathematics, Emory University, Atlanta, USA.
This project has received funding from the European Union’s Horizon 2020 research and innovation programme under the Marie Skłodowska-Curie grant agreement No. 800607.
}
}
\date{\today}

\begin{document}


\maketitle

\begin{abstract}
A graph $H$ is \emph{common} if the number of monochromatic copies of $H$ in a 2-edge-colouring of the complete graph $K_n$ is asymptotically minimised by the random colouring.
Burr and Rosta, extending a famous conjecture of Erd\H{o}s, conjectured that every graph is common. The conjectures of Erd\H{o}s and of Burr and Rosta were disproved by Thomason and by Sidorenko, respectively, in the late 1980s.
Collecting new examples of common graphs had not seen much progress since then,
although very recently a few more graphs were verified to be common by the flag algebra method or the recent progress on Sidorenko's conjecture.

Our contribution here is to provide several new classes of tripartite common graphs.
The first example is the class of so-called triangle-trees, which generalises two theorems by Sidorenko and answers a question of Jagger, \v{S}\v{t}ov\'{i}\v{c}ek, and Thomason from 1996.
We also prove that, somewhat surprisingly, given any tree $T$, there exists a triangle-tree such that the graph obtained by adding $T$ as a pendant tree is still common. 
Furthermore, we show that adding arbitrarily many apex vertices to any connected bipartite graph on at most $5$ vertices yields a common graph. 
\end{abstract}

\section{Introduction}
Ramsey's theorem states that for a fixed graph $H$, every 2-edge-colouring of $K_n$ contains a monochromatic copy of $H$ whenever $n$ is large enough. Perhaps one of the most natural questions extending Ramsey's theorem is how many monochromatic copies of $H$ can be guaranteed to exist. 
To formalise this question, let the \emph{Ramsey multiplicity} $M(H;n)$ be the minimum number of labelled monochromatic copies of $H$ over all 2-edge-colourings of $K_n$. We define the \emph{Ramsey multiplicity constant} $C(H)$ as
\begin{align*}
    C(H):=\lim_{n\rightarrow\infty}\frac{M(H,n)}{n(n-1)\cdots(n-v+1)} = \lim_{n\rightarrow\infty} {M(H,n)} \cdot n^{-v},
\end{align*}
where $v$ is the number of vertices in $H$. 
A random 2-edge-coloring of $K_n$ shows  $C(H)\leq 2^{1-e(H)}$.
We say a graph is \emph{common} if $C(H)=2^{1-e(H)}$. For example, Goodman's formula~\cite{G59} implies that a triangle is common, i.e., $C(K_3)=1/4$.

In 1962, Erd\H{o}s~\cite{E62common} conjectured that every complete graph $K_t$ is common. This was later generalised by Burr and Rosta~\cite{BR80}, who conjectured that in fact every graph $H$ is common. In the late 1980s, both conjectures were disproved.
Sidorenko~\cite{Sid89common} proved that a triangle plus a pendant edge is an uncommon graph, and Thomason~\cite{Thom89} proved that $K_t$ is uncommon for $t\geq 4$.

Since then more examples of uncommon graphs have been found. For instance, Jagger, \v{S}\v{t}ov\'{i}\v{c}ek and Thomason~\cite{JST96} proved that every graph containing $K_4$ as a subgraph is uncommon, and Fox~\cite{F08} proved that $C(H)$ can be exponentially smaller than the commonality bound $2^{1-e(H)}$.

\medskip

Despite many results on the topic, the full classification of common graphs is still a wide open problem. 
All known examples of bipartite common graphs connect with progress on Sidorenko's conjecture~\cite{Sid92}, since the conjecture implies that every bipartite graph is common. 
The converse is also an open question --- does every bipartite common graph satisfies Sidorenko's conjecture? 
Very recently it was shown~\cite{KNNVW20} that a bipartite graph satisfies Sidorenko's conjecture if and only if it is common in \emph{any} multi-colour sense. There has been some progress on Sidorenko's conjecture (see, for example,~\cite{CL20} and references therein) but the full conjecture remains open.

There are not many non-bipartite graphs known to be common.
For example, one of the earliest applications of the flag algebra method established that the 5-wheel is common~\cite{HHKNR12}. 
In case of tripartite graphs, a few more examples have been collected, e.g., odd cycles~\cite{Sid89common} and even wheels~\cite{JST96,Sid96}. 

Two examples of general classes of non-bipartite common graphs are triangle-vertex-trees and triangle-edge-trees, obtained by Sidorenko~\cite{Sid96} and reproved by Jagger, \v{S}\v{t}ov\'{i}\v{c}ek, and Thomason~\cite{JST96}. These can be described recursively. A single triangle is a \emph{triangle-tree} and one may obtain a triangle-tree by identifying a single vertex or an edge of a new triangle with a vertex or an edge, respectively, in a triangle-tree.
A triangle-tree is a \emph{triangle-vertex-tree} (resp. \emph{triangle-edge-tree}) if it is obtained by identifying only vertices (resp. edges). See Figure~\ref{fig:trees} for examples.

\newcommand{\maketriangle}[3]{
\draw (#1) coordinate (#31)  --
++(#2+0:1) coordinate (#32){} --
++(#2+120:1) coordinate (#33){} --cycle
(#31) node[vtx]{}
(#32) node[vtx]{}
(#33) node[vtx]{}
;
}
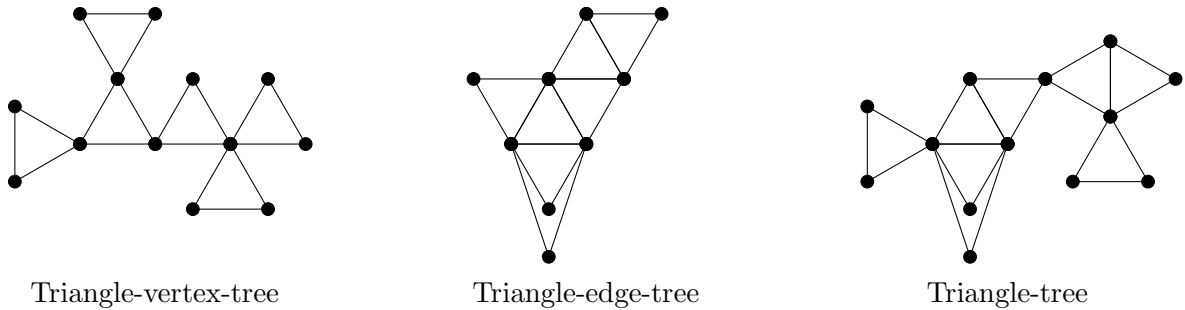
\begin{figure}
{\hskip 1em
    \begin{tikzpicture}
    \maketriangle{0,0}{0}{x}
    \maketriangle{x1}{150}{a}
    \maketriangle{x3}{60}{b}
    \maketriangle{x2}{0}{c}
    \maketriangle{c2}{240}{y}
    \maketriangle{c2}{0}{y}
    \draw
    (1,-2) node{Triangle-vertex-tree}
    ;
    \end{tikzpicture}
\hfill
    \begin{tikzpicture}
    \maketriangle{0,0}{0}{x}
    \maketriangle{x1}{60}{a}
    \maketriangle{x2}{60}{b}
    \maketriangle{b3}{0}{c}
    \maketriangle{c2}{60}{z}
    \maketriangle{x2}{180}{f}
    \draw (0.5,-1.5) node[vtx](x){} (x1)--(x)--(x2)
    ;
    \draw
    (1,-2) node{Triangle-edge-tree}
    ;
    \end{tikzpicture}
\hfill
    \begin{tikzpicture}
    \maketriangle{0,0}{0}{x}
    \maketriangle{x1}{150}{a}
    \maketriangle{x2}{60}{b}
    \maketriangle{b2}{-30}{c}
    \maketriangle{c2}{30}{z}
    \maketriangle{c2}{240}{z}
    \maketriangle{x2}{180}{f}
    \draw (0.5,-1.5) node[vtx](x){} (x1)--(x)--(x2)
    ;
    \draw
    (1,-2) node{Triangle-tree}
    ;
    \end{tikzpicture}    
\hskip 1em}
    \caption{Examples of a triangle-vertex-tree, triangle-edge-tree, and triangle-tree.}
    \label{fig:trees}
\end{figure}

Jagger, \v{S}\v{t}ov\'{i}\v{c}ek and Thomason~\cite{JST96} asked whether tree-like structures other than triangle-vertex (or triangle-edge) trees formed from triangles are common. In particular, they asked if the triangle-tree formed by three triangles, as described in Figure~\ref{fig:jjt}, is common.
We ultimately answer these questions.

\begin{figure}[ht]
\centering
    \begin{tikzpicture}
    \maketriangle{0,0}{-30}{x}
    \maketriangle{x1}{150}{a}
    \maketriangle{x2}{30}{b}
    ;
    \end{tikzpicture}    
    \caption{A triangle-tree suggested by Jagger, \v{S}\v{t}ov\'{i}\v{c}ek and Thomason~\cite{JST96}.}
    \label{fig:jjt}
\end{figure}
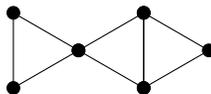

\begin{theorem}\label{thm:tritree}
Every triangle-tree is common.
\end{theorem}

Non-bipartite graphs are more likely to be uncommon than bipartite graphs in many ways. 
Firstly, no examples of bipartite uncommon graphs, which would disprove Sidorenko's conjecture, are known. 
Additionally, Fox~\cite[Lemma~2.1]{F08} observed that any graph with chromatic number at least four and small enough average degree is always uncommon. 
And most importantly, there is a well-known strategy \cite[Theorem~4]{JST96} to produce non-bipartite uncommon graphs. That is, by adding a (possibly large) pendant tree, e.g., a long path, to a non-bipartite graph.
Sidorenko's counterexample, the triangle plus a pendant edge, for the Burr--Rosta conjecture can be seen as one of the earliest examples of this kind.

Our second result states that for some tripartite graphs this strategy of adding a pendant tree fails when adding a small pendant tree.
In other words, there are tripartite graphs that are `robustly common' in the sense that adding any tree of bounded size does not break their commonality.
For a tree~$T$ and a graph $H$, let $T*_{u}^{v}H$ be the graph obtained by identifying $u\in V(T)$ and $v\in V(H)$.

\begin{theorem}\label{thm:adding_tree}
Let $t$ be a positive integer.
If $H$ is a triangle-tree with $2e(H)-3v(H)+3 \geq t$ then $T*_{u}^{v}H$ is common for every choice of tree $T$ with $e(T)\leq t$, $u\in V(T)$, and $v\in V(H)$.
\end{theorem}

Julia Wolf, during her plenary talk at the Canadian Discrete and Algorithmic Mathematics Conference in 2017 on the results from~\cite{SW17}, prompted to complete the list of connected common graphs on five vertices; 
Figure~\ref{fig:smallcommon} depicts the four graphs that had an unknown status at the time of her talk. 

\begin{figure}[ht]
{\hfill
    \begin{tikzpicture}[scale=1.0]
    \draw 
    (0,0) node[vtx](x){}
    (1,0) node[vtx](y){}
    (60:1) node[vtx](z){}
    (y)--++(60:1) node[vtx](w){}--(z)
    (x)--(y)--(z)--(x)
    (2,-0) node[vtx](xx){}
    (w)--(xx);
    \draw(1,-0.5) node{$H_1$};
    \end{tikzpicture}
\hfill
    \begin{tikzpicture}[scale=1.0]
    \draw 
    (0,0) node[vtx](x){}
    (1,0) node[vtx](y){}
    (60:1) node[vtx](z){}
    (y)--++(60:1) node[vtx](w){}--(z)
    (x)--(y)--(z)--(x)
    (2,0) node[vtx](xx){}
    (y)--(xx);
    \draw(1,-0.5) node{$H_2$};
    \end{tikzpicture}
\hfill
    \begin{tikzpicture}[scale=0.7]
    \draw 
    \foreach \i in  {0,1,2,3,4}{
    (90+72*\i:1) node[vtx](v\i){}
    }
    (v1)--(v2)--(v3)--(v4)--(v0)--(v1)--(v4);
    \draw(0,-1.5) node{$H_3$};
    \end{tikzpicture}
\hfill
    \begin{tikzpicture}[scale=0.8]
    \draw 
    \foreach \i in  {1,2,3}{(90+120*\i:1) node[vtx](v\i){}}
    (0,0) node[vtx](v0){}
    (v0)--(v3) node[midway,vtx]{}
    (v1)--(v2)--(v3)--(v1)--(v0)--(v2);
    \draw(0,-1) node{$H_4$};
    \end{tikzpicture}
\hfill}
\caption{Wolf's list of $5$-vertex connected graphs that were not known to be (un)common in 2017.}
\label{fig:smallcommon}
\end{figure}
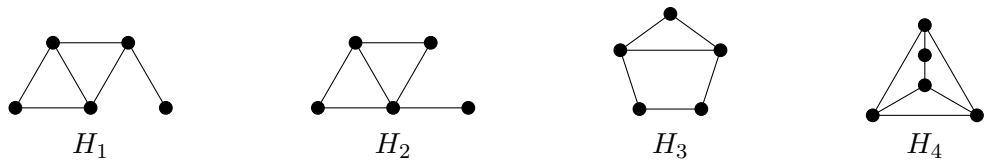

Theorem~\ref{thm:adding_tree} proves that $H_1$ and $H_2$ are common.
The graph $H_3$ was proven to be common in an RSI project at MIT~\cite{RagThesis} using flag algebras.
Another flag algebra application shows that $H_4$ is common; in the Appendix, we give a proof that both $H_3$ and $H_4$ are common.
This completely resolves Wolf's question. 

Analogous applications of flag algebras allow us to show that also various $4$-chromatic graphs are common.
Specifically, we prove that the $7$-wheel as well as all the connected $7$-vertex $K_4$-free non-$3$-colourable graphs are common; see Figure~\ref{fig:4colcommon} for their complete list.
Note that the only previously known examples of non-$3$-colorable graphs were the $5$-wheel or graphs constructed from gluing copies of the $5$-wheel.
We suspect that all the odd wheels except $K_4$ are common, although the plain flag algebra approach for the $9$-wheel is already beyond our computational capacity.

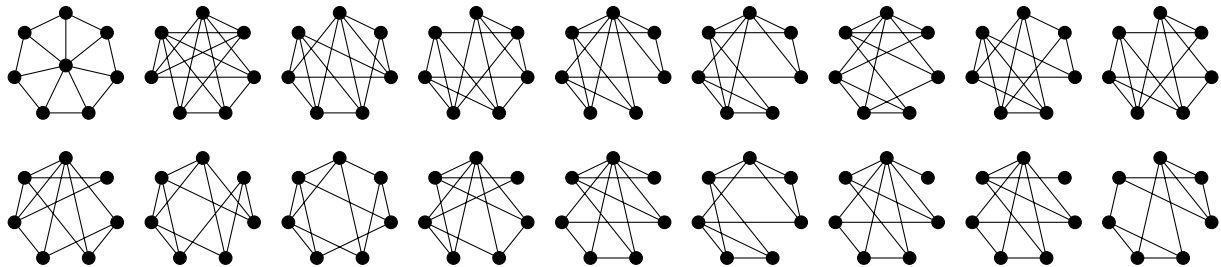
\begin{figure}[ht]\def\FOURcolcommscale{0.7}
{\hfill
    \begin{tikzpicture}[scale=\FOURcolcommscale]
    \draw     (0,0) node[vtx](x){}
    \foreach \i in  {0,1,2,3,4,5,6}{
    (x)--(90+51.428*\i:1) node[vtx](v\i){}
    }
    (v1)--(v2)--(v3)--(v4)--(v5)--(v6)--(v0)--(v1);
    \end{tikzpicture}
\hfill
\begin{tikzpicture}[scale=\FOURcolcommscale]
\draw\foreach \i in {0,1,2,3,4,5,6}{(90+51.4286*\i:1) node[vtx](v\i){}};
\draw(v0)--(v2)(v1)--(v2)(v0)--(v3)(v1)--(v3)(v0)--(v4)(v1)--(v4)(v3)--(v4)(v0)--(v5)(v1)--(v5)(v2)--(v5)(v4)--(v5)(v0)--(v6)(v1)--(v6)(v2)--(v6); \draw (v3)--(v6);
\end{tikzpicture}
\hfill
\begin{tikzpicture}[scale=\FOURcolcommscale]
\draw\foreach \i in {0,1,2,3,4,5,6}{(90+51.4286*\i:1) node[vtx](v\i){}};
\draw(v0)--(v2)(v1)--(v2)(v0)--(v3)(v1)--(v3)(v2)--(v3)(v0)--(v4)(v1)--(v4)(v3)--(v4)(v0)--(v5)(v1)--(v5)(v2)--(v5)(v0)--(v6)(v4)--(v6)(v5)--(v6);
\end{tikzpicture}
\hfill
\begin{tikzpicture}[scale=\FOURcolcommscale]
\draw\foreach \i in {0,1,2,3,4,5,6}{(90+51.4286*\i:1) node[vtx](v\i){}};
\draw(v1)--(v2)(v0)--(v3)(v1)--(v3)(v2)--(v3)(v0)--(v4)(v1)--(v4)(v2)--(v4)(v0)--(v5)(v2)--(v5)(v4)--(v5)(v0)--(v6)(v1)--(v6)(v3)--(v6)(v5)--(v6);
\end{tikzpicture}
\hfill
\begin{tikzpicture}[scale=\FOURcolcommscale]
\draw\foreach \i in {0,1,2,3,4,5,6}{(90+51.4286*\i:1) node[vtx](v\i){}};
\draw(v0)--(v1)(v0)--(v2)(v0)--(v3)(v1)--(v3)(v2)--(v3)(v0)--(v4)(v1)--(v4)(v2)--(v4)(v0)--(v5)(v2)--(v5)(v0)--(v6)(v1)--(v6)(v5)--(v6);
\end{tikzpicture}
\hfill
\begin{tikzpicture}[scale=\FOURcolcommscale]
\draw\foreach \i in {0,1,2,3,4,5,6}{(90+51.4286*\i:1) node[vtx](v\i){}};
\draw(v0)--(v1)(v0)--(v2)(v0)--(v3)(v1)--(v3)(v2)--(v3)(v1)--(v4)(v2)--(v4)(v3)--(v4)(v0)--(v5)(v2)--(v5)(v0)--(v6)(v1)--(v6)(v5)--(v6);
\end{tikzpicture}
\hfill
\begin{tikzpicture}[scale=\FOURcolcommscale]
\draw\foreach \i in {0,1,2,3,4,5,6}{(90+51.4286*\i:1) node[vtx](v\i){}};
\draw(v0)--(v1)(v0)--(v2)(v0)--(v3)(v2)--(v3)(v1)--(v4)(v2)--(v4)(v3)--(v4)(v0)--(v5)(v1)--(v5)(v3)--(v5)(v0)--(v6)(v1)--(v6)(v2)--(v6);
\end{tikzpicture}
\hfill
\begin{tikzpicture}[scale=\FOURcolcommscale]
\draw\foreach \i in {0,1,2,3,4,5,6}{(90+51.4286*\i:1) node[vtx](v\i){}};
\draw(v0)--(v2)(v1)--(v2)(v0)--(v3)(v1)--(v3)(v0)--(v4)(v1)--(v4)(v2)--(v4)(v3)--(v4)(v1)--(v5)(v2)--(v5)(v0)--(v6)(v3)--(v6)(v5)--(v6);
\end{tikzpicture}
\hfill
\begin{tikzpicture}[scale=\FOURcolcommscale]
\draw\foreach \i in {0,1,2,3,4,5,6}{(90+51.4286*\i:1) node[vtx](v\i){}};
\draw(v1)--(v2)(v0)--(v3)(v1)--(v3)(v2)--(v3)(v0)--(v4)(v1)--(v4)(v2)--(v4)(v0)--(v5)(v2)--(v5)(v4)--(v5)(v0)--(v6)(v1)--(v6)(v3)--(v6);
\end{tikzpicture}
\hfill}
\vskip 1em
{\hfill
\begin{tikzpicture}[scale=\FOURcolcommscale]
\draw\foreach \i in {0,1,2,3,4,5,6}{(90+51.4286*\i:1) node[vtx](v\i){}};
\draw(v0)--(v1)(v0)--(v2)(v1)--(v2)(v0)--(v3)(v2)--(v3)(v0)--(v4)(v1)--(v4)(v0)--(v5)(v3)--(v5)(v4)--(v5)(v1)--(v6)(v2)--(v6);
\end{tikzpicture}
\hfill
\begin{tikzpicture}[scale=\FOURcolcommscale]
\draw\foreach \i in {0,1,2,3,4,5,6}{(90+51.4286*\i:1) node[vtx](v\i){}};
\draw(v0)--(v1)(v0)--(v2)(v1)--(v2)(v1)--(v3)(v2)--(v3)(v0)--(v4)(v2)--(v4)(v0)--(v5)(v1)--(v5)(v3)--(v6)(v4)--(v6)(v5)--(v6);
\end{tikzpicture}
\hfill
\begin{tikzpicture}[scale=\FOURcolcommscale]
\draw\foreach \i in {0,1,2,3,4,5,6}{(90+51.4286*\i:1) node[vtx](v\i){}};
\draw(v0)--(v1)(v0)--(v2)(v1)--(v2)(v1)--(v3)(v2)--(v3)(v0)--(v4)(v2)--(v4)(v1)--(v5)(v3)--(v5)(v0)--(v6)(v4)--(v6)(v5)--(v6);
\end{tikzpicture}
\hfill
\begin{tikzpicture}[scale=\FOURcolcommscale]
\draw\foreach \i in {0,1,2,3,4,5,6}{(90+51.4286*\i:1) node[vtx](v\i){}};
\draw(v0)--(v1)(v0)--(v2)(v0)--(v3)(v1)--(v3)(v2)--(v3)(v0)--(v4)(v2)--(v4)(v0)--(v5)(v1)--(v5)(v4)--(v5)(v1)--(v6)(v2)--(v6);
\end{tikzpicture}
\hfill
\begin{tikzpicture}[scale=\FOURcolcommscale]
\draw\foreach \i in {0,1,2,3,4,5,6}{(90+51.4286*\i:1) node[vtx](v\i){}};
\draw(v0)--(v1)(v0)--(v2)(v0)--(v3)(v2)--(v3)(v0)--(v4)(v1)--(v4)(v3)--(v4)(v0)--(v5)(v1)--(v5)(v2)--(v5)(v0)--(v6)(v1)--(v6);
\end{tikzpicture}
\hfill
\begin{tikzpicture}[scale=\FOURcolcommscale]
\draw\foreach \i in {0,1,2,3,4,5,6}{(90+51.4286*\i:1) node[vtx](v\i){}};
\draw(v0)--(v1)(v0)--(v2)(v1)--(v3)(v2)--(v3)(v1)--(v4)(v2)--(v4)(v3)--(v4)(v0)--(v5)(v2)--(v5)(v0)--(v6)(v1)--(v6)(v5)--(v6);
\end{tikzpicture}
\hfill
\begin{tikzpicture}[scale=\FOURcolcommscale]
\draw\foreach \i in {0,1,2,3,4,5,6}{(90+51.4286*\i:1) node[vtx](v\i){}};
\draw(v0)--(v1)(v0)--(v2)(v0)--(v3)(v2)--(v3)(v0)--(v4)(v1)--(v4)(v3)--(v4)(v0)--(v5)(v1)--(v5)(v2)--(v5)(v0)--(v6);
\end{tikzpicture}
\hfill
\begin{tikzpicture}[scale=\FOURcolcommscale]
\draw\foreach \i in {0,1,2,3,4,5,6}{(90+51.4286*\i:1) node[vtx](v\i){}};
\draw(v0)--(v1)(v0)--(v2)(v0)--(v3)(v2)--(v3)(v0)--(v4)(v1)--(v4)(v3)--(v4)(v0)--(v5)(v1)--(v5)(v2)--(v5)(v1)--(v6);
\end{tikzpicture}
\hfill
\begin{tikzpicture}[scale=\FOURcolcommscale]
\draw\foreach \i in {0,1,2,3,4,5,6}{(90+51.4286*\i:1) node[vtx](v\i){}};
\draw(v1)--(v2)(v0)--(v3)(v2)--(v3)(v0)--(v4)(v2)--(v4)(v3)--(v4)(v0)--(v5)(v1)--(v5)(v0)--(v6)(v1)--(v6)(v5)--(v6);
\end{tikzpicture}
\hfill}
\caption{The $7$-wheel and all connected non-$3$-colourable common graphs on 7 vertices.}\label{fig:4colcommon}
\end{figure}

\medskip

Another interesting class of tripartite common graphs was obtained by Sidorenko~\cite{Sid96}.
If a connected bipartite graph $H$ satisfies Sidorenko's conjecture, then adding an \emph{apex} vertex~$v$, i.e., adding all the edges between the new vertex~$v$ and each vertex of $H$, gives a tripartite common graph. 
We conjecture that adding more apex vertices still produces common graphs. For a graph $H$ and a positive integer $a$, let $H^{+a}$ be the graph obtained from $H$ by adding $a$ additional vertices, each new vertex  fully connected to $H$ and not connected to any other new vertex. 
\begin{conjecture}\label{conj:Krst}
If a connected bipartite graph $H$ satisfies Sidorenko's conjecture, then for every positive integer $a$ the graph $H^{+a}$ is common.
In particular, every complete tripartite graph $K_{r,s,t}$ is common.
\end{conjecture}

We verify this conjecture for all connected bipartite graphs $H$ on at most 5 vertices, so, in particular, the complete tripartite graphs $K_{2,2,a}$ and $K_{2,3,a}$ are common for every $a\geq 1$.
\begin{theorem}\label{thm:Krst}
For every connected bipartite graph $H$ on at most $5$ vertices and positive integer $a$ the graph $H^{+a}$ is common.
\end{theorem}

The proof of Theorem~\ref{thm:Krst} relies on the computer-assisted flag algebra method, but we also give a computer-free proof for some cases. 
In particular, we prove without using computers that the octahedron graph, i.e., $C_4^{+2}=K_{2,2,2}$, is common, and generalise it to the so-called beachball graphs $C_{2k}^{+2}$ for every $k\geq2$ (see Theorem~\ref{thm:Bk}). 

\section{Preliminaries}
A \emph{graph homomorphism} from a graph $H$ to a graph $G$ is a vertex map that preserves adjacency.
Let $\Hom(H,G)$ denote the set of all homomorphisms from $H$ to $G$ and let $t_{H}(G)$ be the probability that a uniform random mapping from $H$ to $G$ is a homomorphism. i.e., 
$t_{H}(G)=\frac{|\Hom(H,G)|}{v(G)^{v(H)}}$.

The graph homomorphism density $t_H(G)$ naturally extends to weighted graphs and their limit object \emph{graphons}, i.e., measurable symmetric functions $W:[0,1]^2\rightarrow[0,1]$. We define
\begin{align*}
    t_H(W) := \E \left( \prod_{uv \in E(H)} W(x_u,x_v) \right),
\end{align*}
where $\E$ denotes the integration with respect to the Lebesgue measure on $[0,1]^{v(H)}$.
One may see that the original definition of $t_H(G)$ corresponds to the case $W=W_G$, where $W_G$ is the block 0-1 graphon constructed by the adjacency matrix of $G$.
As nonnegativity of $W$ is unnecessary for the definition, we shall also use $t_H(U):= \E \left( \prod_{uv \in E(H)} U(x_u,x_v) \right)$ for \emph{signed graphons} $U$, i.e., measurable symmetric functions $U:~[0,1]^2\rightarrow[-1,1]$. 

Given a graphon $W$, a $W$-random graph of order $n$ is a graph obtained from $W$ by sampling $n$ points from $[0,1]$ independently and uniformly at random, associating each point with one of the $n$ vertices, and joining two vertices $x,y \in[0,1]$ by an edge with probability $W(x,y)$. It can be proven (see, for example,~\cite{L12}) that if $G_n$ is a $W$-random graph on $n$ vertices, then for every graph $H$ the homomorphism density $t_H(G_n)$ converges to $t_H(W)$ with probability one.

The number of monochromatic copies of a graph $H$ in any 2-edge-colouring of a complete graph can be viewed as the number of copies of $H$ in the graph formed by edges in the first color summed up with the number of copies of $H$ in its complement. Similarly, the density of monochromatic (labelled) copies of $H$ in a 2-edge-colouring can be rewritten as
\begin{align*}
    m_H(W) := t_H(W)+t_H(1-W).
\end{align*}
Note that $m_H(W)=m_H(1-W)$ and $C(H)=\min_W m_H(W)$, where the minimum is taken over all graphons $W$. Indeed, the minimum exists by the compactness of the space of graphon under the cut norms and the latter follows from considering the $W$-random graphs explained in the previous paragraph. Thus, a graph $H$ is common if and only if $m_H(W)\geq 2^{1-e(H)}$ for each graphon $W$.

\medskip

 Let $\ev(H)$ be the family of subgraphs of $H$ with even number of edges and let $\ev_+(H)$ be the collection of nonempty graphs in $\ev(H)$. 
 Then, with $U:=2W-1$,
 \begin{align}\label{eq:even}
     m_H(W) &=  t_H\left(\frac{1+U}{2}\right) +t_H\left(\frac{1-U}{2}\right)\nonumber \\
     &=2^{1-e(H)}\sum_{F\in \ev(H)} t_F(U) = 2^{1-e(H)}\left(1 + \sum_{F\in \ev_+(H)}t_F(U)\right). 
\end{align}
Hence, $H$ is common if and only if $\sum_{F\in \ev_+(H)}t_F(U)\geq 0$ for every signed graphon $U$.

An immediate consequence of this expansion is a well-known formula by Goodman~\cite{G59}.
\begin{lemma}[Goodman's formula]\label{lem:goodman}
For every graphon $W$, $m_{K_3}(W)=\frac{3}{2}m_{K_{1,2}}(W)-\frac{1}{2}$.
\end{lemma}
\begin{proof}
By~\eqref{eq:even}, $m_{K_3}(W) = \frac{3}{4}t_{K_{1,2}}(U) +\frac{1}{4}$ and $m_{K_{1,2}}(W) = \frac{1}{2}t_{K_{1,2}}(U)+\frac{1}{2}$.
\end{proof}

\medskip

The following is an easy consequence of H\"older's inequality, which will be repeatedly used.
\begin{lemma}\label{lem:tm}
Let $H,F$, and $J$ be graphs, $W$ a graphon, and $k$ and $\ell$ positive integers with $\ell\geq k$.
If
\[
t_H(W)\geq \frac{t_J(W)^\ell}{t_F(W)^{k-1}} \text{\,\,\,\, and \,\,\,\,} t_H(1-W)\geq \frac{t_J(1-W)^\ell}{t_F(1-W)^{k-1}},
\]
then
\begin{align*}
    m_H(W) \geq 2^{k-\ell}\frac{m_J(W)^\ell}{m_F(W)^{k-1}}.
\end{align*}
\end{lemma}
\begin{proof}
We use H\"older's inequality of the form
\begin{align*}
    \prod_{i=1}^k\int f_i(x)^k dx \geq \left(\int \prod_{i=1}^k f_i(x) dx \right)^k
\end{align*}
for nonnegative functions $f_i$.
Let the integration be the sum of two terms. Then 
\begin{align}\label{eq:holder}
    \prod_{i=1}^{k} (a_i^k+b_i^k) \geq \left(\prod_{i=1}^{k}a_i +\prod_{j=1}^k b_j\right)^k
\end{align}
for nonnegative numbers $a_i$ and $b_j$, it follows that
\begin{align*}
    m_H(W) &= t_H(W)+t_H(1-W) \geq \frac{t_J(W)^\ell}{t_F(W)^{k-1}}+\frac{t_J(1-W)^\ell}{t_F(1-W)^{k-1}}\\
    &= \left(\frac{t_J(W)^\ell}{t_F(W)^{k-1}}+\frac{t_J(1-W)^\ell}{t_F(1-W)^{k-1}}\right)\big(t_F(W)+t_F(1-W)\big)^{k-1} m_F(W)^{-k+1}\\
    &\geq_{\eqref{eq:holder}} \left(t_J(W)^{\frac{\ell}{k}} +t_J(1-W)^{\frac{\ell}{k}}\right)^{k}m_F(W)^{-k+1}\\
    &\geq 2^{k-\ell}\frac{m_J(W)^{\ell}}{m_F(W)^{k-1}}.
\end{align*}
Indeed, the first inequality is H\"older's inequality~\eqref{eq:holder} and the second follows from convexity of the function $f(z)=z^{\ell/k}$, as $\ell\geq k$.
\end{proof}

For the proof of Theorem~\ref{thm:adding_tree}, we take an information-theoretic approach. We state the following fact about entropy without proof and refer the reader to \cite{AS08} for more detailed information on entropy and conditional entropy.
 
\begin{lemma}\label{lem:entropy}
Let $X$ be a random variable taking values in a set $S$ and let $\HH(X)$ be the entropy of~$X$. Then $\HH(X)\leq\log|S|$.
\end{lemma}

\section{Triangle-trees}
To describe triangle-trees, it is convenient to use the notion of tree decompositions, introduced by Halin~\cite{Hal76}
and developed by Robertson and Seymour \cite{RS84}.
\begin{definition}
A \emph{tree-decomposition} of a graph $H$ is a pair $(\mathcal{F}, \TT)$ consisting of a family $\mathcal{F}$  of vertex-subsets of $H$ and  a tree $\TT$ with $V(\TT)=\mathcal{F}$ such that
\begin{enumerate}
\item $\bigcup_{X\in\mathcal{F}}X=V(H)$, 
\item for each $e \in E(H)$, there exists a set $X \in \mathcal{F}$ such that
$X$ contains $e$, and
\item for $X,Y,Z\in \mathcal{F}$, $X\cap Y\subseteq Z$ 
whenever $Z$ lies on the path from $X$ to $Y$ in $\TT$.
\end{enumerate}
\end{definition}
Following~\cite{CL16,L19}, we say that $H$ is a \emph{$J$-tree} if and only if there exists a tree decomposition $(\mathcal{F},\TT)$ such that the subgraph $H[X]$ of $H$ induced on $X\in \FF$ is isomorphic to $J$ and moreover, there is an isomorphism between $H[X]$ and $H[Y]$ that fixes $H[X\cap Y]$ whenever $XY\in E(\TT)$.
Such a tree-decomposition $(\FF,\TT)$ of $H$ is called a \emph{$J$-decomposition}.
When $J=K_3$, we simply say that $H$ is a triangle-tree with a triangle-decomposition $(\FF,\TT)$. It is straightforward to see that this definition is equivalent to the recursive one given in the introduction.

\medskip

For a triangle-tree $H$ with a triangle-decomposition $(\FF,\TT)$, one may easily relate $|\FF|$ to $v(H)$ and $e(H)$. Let $\varphi(H) := e(H) - v(H) + 1$ and $\kappa(H) := 2e(H) - 3v(H) + 3$.
\begin{lemma}\label{lem:tree_identities}
If $H$ is a triangle-tree with a triangle-decomposition $(\FF,\TT)$, then
$|\FF|=\varphi(H)$ and the number of edges $XY\in E(\TT)$ such that the subgraph $H[X\cap Y]$ is a single edge equals to $\kappa(H)$.
In particular, $\kappa(H) \le \varphi(H) - 1$ for every triangle-tree $H$.
\end{lemma}
\begin{proof}
Let $k:=k(\FF)$ be the number of edges $XY\in E(\TT)$ such that the subgraph $H[X\cap Y]$ is an edge.
For an edge $e\in E(H)$, let $t_e$ be the number of contributions of $e$ in the sum 
$\sum_{X\in\FF} e(H[X])$. That is,
\begin{align*}
    3|\FF|=\sum_{X\in\FF} e(H[X]) = \sum_{e\in E(H)} t_e.
\end{align*}
On the other hand, $t_e-1$ is equal to the number of edges $XY\in E(\TT)$ such that $H[X\cap Y]$ is the single-edge $\{e\}$, which proves $e(H)=3|\FF|-k$. 
Analogously, $v(H)=2|\FF|+1 - k$ and hence, $|\FF|=e(H)-v(H)-1=\varphi(H)$ and $k(\FF)=2e(H)-3v(H)+3=\kappa(H)$. Finally, $\kappa(H) = k \le e(\TT) = |\FF|-1 = \varphi(H) - 1$.
\end{proof}

The key ingredient in the proof of Theorem~\ref{thm:tritree} is the following lemma.
\begin{lemma}[\cite{L19}, Theorem 2.7]\label{lem:tree_hom}
        If $H$ is a $J$-tree with a $J$-decomposition $(\FF,\TT)$ and $W$ is a graphon with $t_{J}(W) > 0$, then
	\begin{align}\label{eq:Jtree_hom}
		t_H(W)\geq \frac{t_{J}(W)^{|\FF|}}{\prod_{XY\in E(\TT)} t_{H[X\cap Y]}(W)}.
	\end{align}
\end{lemma}
Lemma~\ref{lem:tree_hom} is basically just simplifying multiple applications of the Cauchy--Schwarz inequality or Jensen's inequality. 
For example, $K_{1,1,t}$ is a triangle-tree, since there is a triangle decomposition $(\FF,\TT)$ that consists of $|\FF|=t$ and the star $\TT$ on $\FF$ with $t-1$ leaves, where each vertex subset in $\FF$ induces a triangle; see Figure~\ref{fig:treedecomp}. Thus, Lemma~\ref{lem:tree_hom} gives $t_{K_{1,1,t}}(W)\geq t_{K_3}(W)^{t}/t_{K_2}(G)^{t-1}$, which also follows from a standard application of Jensen's inequality.

\begin{figure}[ht]
\centering
    \begin{tikzpicture}
    \draw 
    (0,0) node[vtx,label=above:$x$](x){}
    (1,0) node[vtx,label=above:$y$](y){}
    \foreach \i in  {1,2,3,4}
    {
    (-2+\i,-1) node[vtx,label=below:$v_\i$](v\i){}
    (x)--(v\i)--(y)
    }
    (x)--(y)
    ;
    
\begin{scope}[xshift=8cm]
\draw
(0,0) node[rectangle,draw](XX){$x,y,v_1$}
    \foreach \i in  {2,3,4}  {
    (-6+2*\i,-1) node[rectangle,draw](YY){$x,y,v_\i$}
    (XX)--(YY)
    }
    ;    
;    
\end{scope}    
    \end{tikzpicture}
    \caption{$K_{1,1,4}$ and its tree decomposition $(\FF,\TT)$.}
    \label{fig:treedecomp}
\end{figure}
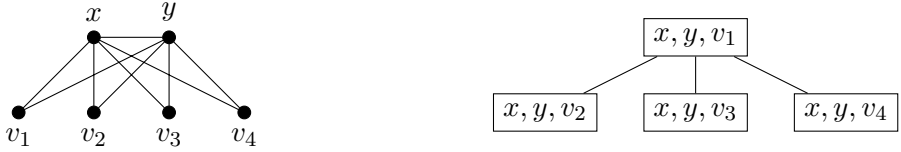

\medskip

In order to prove Theorem~\ref{thm:tritree}, we are going to apply Lemma~\ref{lem:tree_hom} for $J=K_3$.
\begin{corollary}
If $H$ is a triangle-tree and $W$ is a nonzero graphon, then
\begin{align}\label{eq:tree_hom}
	t_H(W) \geq \frac{ t_{K_3}(W)^{\varphi(H)}}{t_{K_2}(W)^{\kappa(H)}}.
\end{align}
\end{corollary}

We are now ready to prove Theorem~\ref{thm:tritree}.

\begin{proof}[Proof of Theorem~\ref{thm:tritree}]
Let $H$ be a triangle-tree. If $W=1$ or $W=0$ almost everywhere, then $m_H(W)=1$. 
Otherwise, two applications of \eqref{eq:tree_hom} yields
\[
t_H(W)\geq \frac{t_{K_3}(W)^{\varphi(H)}}{t_{K_2}(W)^{\kappa(H)}}
\quad \mbox{and}\quad
t_H(1-W)\geq \frac{t_{K_3}(1-W)^{\varphi(H)}}{t_{K_2}(1-W)^{\kappa(H)}}
.\]
Therefore, by Lemma~\ref{lem:tm} with $J=K_3$, $H=K_2$, $\ell=\varphi(H)$ and $k=\kappa(H)+1$, we have
\begin{align*}
    m_H(W) \geq 2^{\kappa(H)+1-\varphi(H)} \cdot \frac{m_{K_3}(W)^{\varphi(H)}}{m_{K_2}(W)^{\kappa(H)}}
    = 2^{\kappa(H)+1-\varphi(H)} \cdot m_{K_3}(W)^{\varphi(H)}
    \geq 2^{\kappa(H)+1-3\varphi(H)}= 2^{1-e(H)}.
\end{align*}
Indeed, the last inequality uses the commonality of a triangle, i.e., $m_{K_3}(W) \geq 1/4$.
\end{proof}

To prove Theorem~\ref{thm:adding_tree}, we need a slightly more careful analysis than just a simple application of Lemma~\ref{lem:tree_hom}. The main tool is~\cite[Theorem~2.6]{L19}, which will be stated shortly.
Let $\FF$ be a family of subsets of $[k]:=\{1,2,\dots,k\}$.
A \emph{Markov tree} on $[k]$ is a pair $(\FF,\TT)$ with $\TT$ a tree on vertex set $\FF$ that satisfies
\begin{enumerate}
	\item $\bigcup_{F\in\mathcal{F}}F=[k]$ and
	\item for $A,B,C\in \mathcal{F}$, $A\cap B\subseteq C$ 
whenever $C$ lies on the path from $A$ to $B$ in $\TT$.
\end{enumerate}
This is an abstract tree-like structure without the graph structure considered in defining tree-decompositions. In particular, a tree-decomposition of $H$ is a Markov tree on $V(H)$. For more detailed explanation, we refer to~\cite{L19}. Let $V$ be a finite set and for each $F\in\FF$ let $\X_F=(X_{i;F})_{i\in F}$ be a random vector 	taking values in~$V^{F}$.
The following theorem states that there exist random variables $Y_1,Y_2,\dots,Y_k$ such that, for each $F\in \FF$, the two random vectors $(Y_i)_{i\in F}$ and $\X_{F}$ are identically distributed over $V^{F}$ and, moreover,
the maximum entropy under such constraints
can always be attained.

\begin{lemma}[\cite{L19}, Theorem~2.6]\label{lem:tree_entropy}
	Let $(\FF,\TT)$ be a Markov tree on $[k]$.
	Let $V$ be a finite set and for each $F\in\FF$ let $\X_F=(X_{i;F})_{i\in F}$ be a random vector 
	taking values in $V^F$.  
	If $(X_{i;A})_{i\in A\cap B}$ and $(X_{j;B})_{j\in A\cap B}$ are identically distributed
	whenever $AB\in E(\TT)$,
	then there exists $\Y=(Y_1,\dots,Y_k)$
	with entropy
	\begin{align}\label{eq:tree_entropy}
		\HH(\Y)
		=\sum_{F\in\FF}\HH(\X_F)
		-\sum_{AB\in E(\TT)}\HH((X_{i;A})_{i\in A\cap B})
	\end{align}
	such that $(Y_i)_{i\in F}$ and $\X_F$ are identically distributed over $V^{F}$ for all $F\in\FF$.
\end{lemma}

An entropy analysis using this lemma give the following corollary.
Recall that for a tree~$T$ and a graph $H$, we denote by  $T*_{u}^{v}H$ the graph obtained by identifying $u\in V(T)$ and $v\in V(H)$.
\begin{lemma}\label{lem:add_tree}
If $H$ is a triangle-tree  and $T$ a tree with at most $\kappa(H)$ edges, then
\begin{align}\label{eq:add_tree}
    t_{T*_{u}^{v}H}(W)\geq \frac{t_{K_3}(W)^{\varphi(H)}}{t_{K_2}(W)^{\kappa(H)-e(T)}}
\end{align}
for every $u\in V(T)$ and $v\in V(H)$.
\end{lemma}
Using this lemma, the proof of Theorem~\ref{thm:adding_tree} is almost identical to that of Theorem~\ref{thm:tritree}.
\begin{proof}[Proof of Theorem~\ref{thm:adding_tree}]
Let $H$ be a triangle-tree such that $\kappa(H) \ge t$, $T$ a tree with at most $t$ edges, and $W$ a nonzero graphon.
By Lemma~\ref{lem:tree_identities}, $e(H)=3\varphi(H)-\kappa(H)$ and thus, \[e(T*_{u}^{v}H)=e(T)+e(H)=3\varphi(H)-\kappa(H)+e(T).\]
Combining \eqref{eq:add_tree} and Lemma~\ref{lem:tm} for $J=K_3$, $H=K_2$, $\ell=\varphi(H)$ and $k=\kappa(H)-e(T)+1$ yields
\[
    m_{T*_{u}^{v}H}(W) \geq 2^{\kappa(H)+1-e(T)-\varphi(H)} \cdot m_{K_3}(W)^{\varphi(H)}.
\]
Note that in order to apply Lemma~\ref{lem:tm}, we required $\kappa(H) \ge e(T)$.
As $m_{K_3}(W)\geq 1/4$, we have
\[
     m_{T*_{u}^{v}H}(W) \geq 2^{\kappa(H)+1-e(T)-3\varphi(H)} = 
     2^{1-e(H)-e(T)} = 2^{1-e(T*_{u}^{v}H)}.
\]
Therefore, $T*_{u}^{v}H$ is common.
\end{proof}

It remains to prove Lemma~\ref{lem:add_tree}.
\begin{proof}[Proof of Lemma~\ref{lem:add_tree}]
Let $(\FF,\TT)$ be a triangle-decomposition of $H$ and $k:=\kappa(H)$. Recall that $k$ is 
the number of edges $XY \in E(\TT)$ such that the subgraph $H[X \cap Y] \cong K_2$.
The first step is to find a natural tree-decomposition of $T*_{u}^{v}H$ that extends $(\FF,\TT)$.

Let $T$ be rooted at a leaf $x\in V(T)$ and suppose that we orient each edge of $T$ away from the root. Let $\S$ be a tree on $E(T)$, where the oriented edges $(u_1,v_1)$ and $(u_2,v_2)$ are adjacent if and only if $v_1=u_2$. One may easily check that
$(E(T),\S)$ is a tree decomposition of $T$. 
Now pick an edge $uu'\in E(T)$, which is a vertex of $\S$, and connect it to a vertex bag $X\in \FF$ that contains $v\in V(T)$ while identifying $u$ and $v$. This new tree $\TT'$, obtained by adding an edge between two vertices $uu'$ and $X$, gives a tree-decomposition $(\FF',\TT')$ of $T*_{u}^{v}H$, where $\FF':=V(\TT')=V(\TT)\cup V(\S)$.

Since the homomorphic density in a sequence of $W$-random graphs of increasing sizes converges to the homomorphic density in $W$, as explained in the preliminaries, it is enough to prove the inequality~\eqref{eq:add_tree} for an $n$-vertex graph $G$ instead of a graphon $W$.
For brevity, we identify the vertex set $V(T*_{u}^{v}H)$ with the set $[t]$ and let $1\in [t]$ be the vertex shared by $H$ and $T$.
For each $F\in \FF'$ with $|F|=3$, let $\X_F=(X_{i;F})_{i\in F}$ be a uniform random triangle in $G$, labelled by vertices in $F$.
If $|F|=2$ then let $\X_F=(X_{i;F})_{i\in F}$ be a random edge labelled by vertices in $F$ sampled in such a way that $\PP[\X_{F}=(v_1,v_2)]$ is proportional to the number of triangles that contains the edge~$v_1v_2\in E(G)$. We call this possibly non-uniform edge distribution \emph{triangle-projected}.

We claim that $(X_{i;A})_{i\in A\cap B}$ and $(X_{i;B})_{i\in A\cap B}$ are identically distributed. If $|A\cap B|=2$, then both distributions are triangle-projected. If $|A\cap B|=1$, i.e., $A\cap B=x\in V(T*_{u}^{v}H)$, then both distributions are proportional to the weighted degree sum $\sum_{x\subset e} p_e$ where $p_e$ is the probability of an edge being sampled by the triangle-projected distribution. Therefore, by Lemma~\ref{lem:tree_entropy}, there exists $\Y=(Y_1,\dots,Y_t)$ with entropy
\begin{align*}
		&\HH(\Y)
		=\sum_{F\in\FF'}\HH(\X_F)
		-\sum_{AB\in E(\TT')}\HH((X_{i;A})_{i\in A\cap B})\\
		&=\sum_{F\in\FF}\HH(\X_F) +\sum_{F\in E(T)}\HH(\X_F)
		-\sum_{AB\in E(\TT)}\HH((X_{i;A})_{i\in A\cap B})
		-\sum_{AB\in E(\S)}\HH((X_{i;A})_{i\in A\cap B})
		-\HH(Y_{1}).
	\end{align*}
Recall that the vertex 1 is the vertex shared by $T$ and $H$, so $Y_1$ means the random image of the vertex with respect to $\Y$.
For $F\in \FF$, $\HH(\X_F)=\log|\Hom(K_3,G)|$, since $\X_F$ is a uniform random triangle.
For $F\in E(T)$, $\HH(\X_F)$ is the entropy $h_e$ of the triangle-projected edge distribution.
There are exactly $k$ cases such that $|A\cap B|=2$ and $AB\in E(\TT)$, and for such cases, $\HH((X_{i;A})_{i\in A\cap B})=h_e$.
Thus, 
\begin{align*}
\HH(\Y)	&=|\FF|\log|\Hom(K_3,G)| + \sum_{F\in E(T)}\HH(\X_F)
		-\sum_{AB\in E(\TT')}\HH((X_{i;A})_{i\in A\cap B})
		\\&\geq |\FF|\log|\Hom(K_3,G)| - (k-e(T))h_e
		- (e(\TT')-k)\log|V(G)| \\
		&\geq  |\FF|\log|\Hom(K_3,G)| - (k-e(T))\log|\Hom(K_2,G)|
		- (e(\TT')-k)\log n
\end{align*}
Indeed, the first inequality follows from the bound $\HH((X_{i;A})_{i\in A\cap B})\leq\log n$ by Lemma~\ref{lem:entropy} when $|A\cap B|=1$, and the second follows from the bound $h_e\leq\log|\Hom(T*_{u}^{v}H,G)|$ by the same lemma.
Again by Lemma~\ref{lem:entropy}, $\HH(\Y)\leq\log |\Hom(H,G)|$. Thus,
\begin{align*}
    t_{T*_{u}^{v}H}(G)&=
    \frac{|\Hom(T*_{u}^{v}H,G)|}{n^{v(H)+v(T)-1}}
     \geq \frac{|\Hom(K_3,G)|^{|\FF|}}{|\Hom(K_2,G)|^{k-e(T)}n^{e(\TT')-k}}\cdot\frac{1}{n^{v(H)+v(T)-1}}\\
     &=\frac{t_{K_3}(G)^{|\FF|}}{t_{K_2}(G)^{k-e(T)}n^{e(\TT')+k-2e(T)}}\cdot\frac{n^{3|\FF|}}{n^{v(H)+v(T)-1}}
     =\frac{t_{K_3}(G)^{|\FF|}}{t_{K_2}(G)^{k-e(T)}},
\end{align*}
where the last equality follows from the identity $e(\TT')=e(\TT)+e(\S)+1 = |\FF|+e(T)-1$ and Lemma~\ref{lem:tree_identities}.
\end{proof}

\section{Beachball graphs and bipartite graphs with apex vertices}

The proof of Theorem~\ref{thm:Krst} combines our novel ideas and the flag algebra method developed by Razborov~\cite{R07}.
To demonstrate how the new method works without using flag algebras, we firstly prove that $K_{2,2,2}$ is common. 
\begin{theorem}\label{thm:K222}
The octahedron $K_{2,2,2}$ is common.
\end{theorem}

By a standard application of the Cauchy--Schwarz inequality (or Lemma~\ref{lem:tree_hom}), it is easy to see that $t_{K_{2,2,2}}(W) \geq t_{K_{1,2,2}}(W)^2/t_{C_4}(W)$. Then by Lemma~\ref{lem:tm}, we immediately obtain 
\begin{align}\label{eq:222v122}
    m_{K_{2,2,2}}(W) \geq \frac{m_{K_{1,2,2}}(W)^2}{m_{C_4}(W)}.
\end{align}
By Sidorenko's theorem~\cite{Sid96}, the 4-wheel $K_{1,2,2}$ is common; however, $m_{C_4}(W)=1/8$ if and only if $W=1/2$ almost everywhere, i.e., $W$ is quasirandom, the naive approach using commonality of~$K_{1,2,2}$ while bounding $m_{C_4}$ from above does not work. 
We circumvent this difficulty by comparing $m_{K_{1,2,2}}(W)$ and $m_{C_4}(W)$. 
Another application of the Cauchy--Schwarz inequality, together with Lemma~\ref{lem:tm}, gives
\begin{align*}
    m_{K_{1,2,2}}(W)\geq \frac{m_{K_{1,1,2}}(W)^2}{m_{K_{1,2}}(W)}.
\end{align*}
For brevity, denote $D:=K_{1,1,2}$, which is the \emph{diamond} graph obtained by adding a diagonal edge to the 4-cycle.
The following lemma, partly motivated by~\cite{HHKNR12}, enables us to compare $m_{D}(W)$ and $m_{C_4}(W)$.
\begin{lemma}\label{lem:diamond}
Let $0\leq c\leq (3-\sqrt{5})/4$.
For any graphon $W$, the following inequality holds:
\begin{align*}
    m_D(W) -1/16 \geq c(m_{C_4}(W)-1/8). 
\end{align*}
\end{lemma}
\begin{proof}
Using~\eqref{eq:even} with $U:=2W-1$ and $H=D$, we obtain
\begin{align*}
    m_D(W) -\frac{1}{16} = \frac{1}{16}\sum_{F\in \ev_+(D)} t_F(U)
    = \frac{1}{16}\big( 2t_{2\cdot K_2}(U)+ 8 t_{K_{1,2}}(U)  +4t_{K_3^+}(U)+t_{C_4}(U)\big),
\end{align*}
where $K_3^+$ denotes the triangle plus a pendant edge.
The same argument for $m_{C_4}(W)$ yields
\begin{align*}
    m_{C_4}(W) - \frac{1}{8} = \frac{1}{8}\big(2 t_{2\cdot K_2}(U)+ 4t_{K_{1,2}}(U) +t_{C_4}(U)\big),
\end{align*}
and thus,
\begin{align}\label{eq:expand}
    m_D(W) - \frac{1}{16} & ~-~ c(m_{C4}(W)-\frac{1}{8})\nonumber \\
    &= \frac{1}{16}\big((2-4c)t_{2\cdot K_2}(U)+(8-8c) t_{K_{1,2}}(U) +4t_{K_3^+}(U)+(1-2c)t_{C_4}(U)\big). 
\end{align}
Recall that $U=2W-1$ is not necessarily nonnegative, but $t_{K_{1,2}}(U)$, $t_{2\cdot K_2}(U)$, and $t_{C_4}(U)$ are always nonnegative, since $t_{K_{1,2}}(U)\geq t_{K_2}(U)^2=t_{2\cdot K_2}(U)$ and $t_{C_4}(U)\geq t_{K_{1,2}}(U)^2$. 
The key inequality we shall prove is
\begin{align}\label{eq:K3+}
    |t_{K_3^+}(U)| \leq \sqrt{t_{K_{1,2}}(U) t_{C_4}(U)}.
\end{align}
Suppose that this is true. Then \eqref{eq:expand} gives the lower bound
\begin{align*}
    (2-4c)t_{2\cdot K_2}(U)+(8-8c) t_{K_{1,2}}(U) -4\sqrt{t_{K_{1,2}}(U) t_{C_4}(U)}+(1-2c)t_{C_4}(U)
\end{align*}
for $ 16\big(m_D(W) -1/16 - c(m_{C4}(W)-1/8)\big)$. 
This is nonnegative whenever $(8-8c)(1-2c)\geq 4$ and $c\leq 1/2$.
Taking $0\leq c\leq \frac{3-\sqrt{5}}{4}$ suffices for this purpose.

It remains to prove~\eqref{eq:K3+}. Denote $\nu(x,z):=\E_{y}U(x,y)U(y,z)$ and $\mu(z):=\E_{w}U(z,w)$. Then
\begin{align*}
    |t_{K_3^+}(U)| &= \left|\E\big[ U(x,y)U(y,z)U(z,x)U(z,w)\big]\right|
    = \left|\E \big[\nu(x,z)U(z,x)\mu(z)\big]\right|\\
    &\leq (\E [\nu(x,z)^2])^{1/2} (\E [U(z,x)^2\mu(z)^2])^{1/2}\\
    &\leq  (\E [\nu(x,z)^2])^{1/2} (\E [\mu(z)^2])^{1/2}=\sqrt{t_{K_{1,2}}(U) t_{C_4}(U)}.\tag*{\qedhere} 
\end{align*}
\end{proof}

\begin{proof}[Proof of Theorem~\ref{thm:K222}]
Recall that repeated applications of Lemma~\ref{lem:tm} yield
\begin{align*}
    m_{K_{2,2,2}}(W) \geq \frac{m_{D}(W)^4}{m_{K_{1,2}}(W)^2 m_{C_4}(W)}.
\end{align*}
By Goodman's formula (Lemma~\ref{lem:goodman}), $m_{K_{1,2}}(W)=\frac{2}{3}m_{K_3}(W)+\frac{1}{3}$.
Together with the inequality $m_{D}(W)\geq m_{K_3}(W)^2$ that follows from Lemma~\ref{lem:tm} and the inequality $t_D(W)\geq t_{K_3}(W)^2/t_{K_2}(W)$,
we obtain
\begin{align}\label{eq:goodmanD}
    m_{K_{1,2}}(W)=\frac{2}{3}m_{K_3}(W)+\frac{1}{3}\leq \frac{2}{3}\sqrt{m_{D}(W)}+\frac{1}{3}.
\end{align}
Therefore, by using Lemma~\ref{lem:diamond},
\begin{align*}
    m_{K_{2,2,2}}(W) \geq \frac{m_{D}(W)^4}{m_{K_{1,2}}(W)^2 m_{C_4}(W)}
    \geq \frac{c\cdot m_{D}(W)^4}{\left(\frac{2}{3}\sqrt{m_{D}(W)}+\frac{1}{3}\right)^2 \left(m_D(W)-\frac{1}{16}+\frac{c}{8}\right)}.
\end{align*}
This lower bound is a rational function $h_c$ of $x:=\sqrt{m_D(W)}$, which simplifies to
\begin{align*}
    h_c(x):= \frac{144cx^8}{(2x+1)^2(16x^2-1+2c)}.
\end{align*}
We are looking at the range $x\geq 1/4$, as $m_{D}(W)\geq 1/16$ by commonality of $D$. 
Taking, for example, $c=1/7<\frac{3-\sqrt{5}}{4}$ makes the function $h_c$ monotone increasing on $x\geq 1/4$, and thus, $h_c(z)\geq f_c(1/4)=2^{-11}$.
This proves that $K_{2,2,2}$ is common.
\end{proof}

\begin{figure}[ht]
{
    \begin{tikzpicture}
    \draw
    \foreach \x in {0,1,...,6}{
    (60*\x:1) node[vtx](x\x){}
    }
    (x0)--(x2)--(x4)--(x0)--(x3)--(x5)--(x0)
    (x1)--(x2)--(x4)--(x1)--(x3)--(x5)--(x1)
    ;
    \draw[line width=1.5pt]
    {
    (x2)--(x4)--(x3)--(x5)--(x2)
    }
    ;
    \draw (0,-1.6)node{$B_4=K_{2,2,2}$};
    \end{tikzpicture}  
\hfill
    \begin{tikzpicture}
    \draw
     (2.5,1) node[vtx](a){}
     (2.5,-1) node[vtx](b){}
    \foreach \x in {0,1,...,5}{
    (a) -- (\x,0) node[vtx](x\x){} -- (b)
    }
    (x0)--(x5)
    (x0) to [bend left=10](x5)
    ;
    \draw (2.5,-1.6)node{$B_6$};
    \end{tikzpicture}
\hfill
    \begin{tikzpicture}
    \draw
     (1,1) node[vtx](a){}
     (1,-1) node[vtx](b){}
    \foreach \x in {0,1,2}{
    (a) -- (\x,0) node[vtx](x\x){} -- (b)
    }
    (x0)--(x2)
    ;
    \draw (1,-1.6)node{$D_2$};
    \end{tikzpicture}    
\hfill
    \begin{tikzpicture}
    \draw
     (2,1) node[vtx](a){}
     (2,-1) node[vtx](b){}
    \foreach \x in {0,1,...,4}{
    (a) -- (\x,0) node[vtx](x\x){} -- (b)
    }
    (x0)--(x4)
    ;
    \draw (2,-1.6)node{$D_4$};
    \end{tikzpicture}    
}
    \caption{Graphs $B_4$, $B_6$, $D_2$, and $D_4$.}
    \label{fig:ball}
\end{figure}
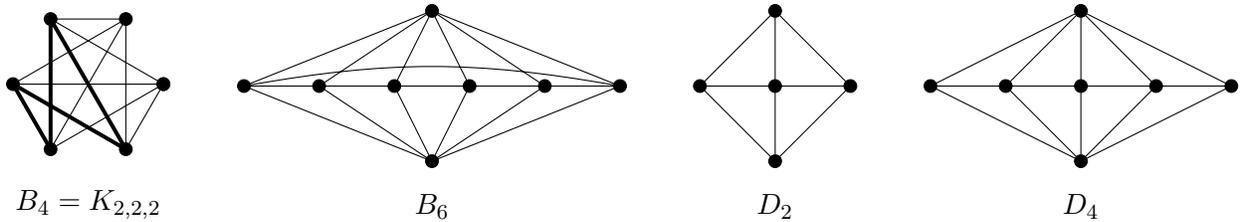

Let the \emph{$k$-beachball graph} $B_k$ be the graph obtained by gluing two copies of $k$-wheels along the $k$-cycle. In particular, $K_{2,2,2}$ is the 4-beachball, since it can be obtained by gluing two copies of 4-wheels along a 4-cycle. See Figure~\ref{fig:ball}, where the 4-cycle is marked bold. 
As a straightforward generalisation of Theorem~\ref{thm:K222}, we also prove the following theorem.

\begin{theorem}\label{thm:Bk}
For every $k\geq 2$, the $2k$-beachball $B_{2k}$ is common.
\end{theorem}
\begin{proof}
The proof is essentially the same as Theorem~\ref{thm:K222} despite a slightly general setting. 
Let $D_k$ be the graph obtained by adding two apex vertices to a $k$-edge path, i.e., it consists of $k$ copies of diamonds glued along $K_{1,2}$'s centred at the vertices of degree three in a path-like way, as described in Figure~\ref{fig:ball}. In particular, $D_1=D$ and $D_2$ is the 4-wheel.
Lemma~\ref{lem:tree_hom} then gives
\begin{align*}
    t_{D_k}(W)\geq \frac{t_D(W)^k}{t_{K_{1,2}}(W)^{k-1}}
\end{align*}
and thus, $m_{D_k}(W)\geq m_D(W)^k/m_{K_{1,2}}(W)^{k-1}$ by Lemma~\ref{lem:tm}.

The $2k$-beachball is then obtained by gluing two copies of $D_k$ along the 4-cycle that contains two vertices of degree three. The standard application of the Cauchy--Schwarz inequality (or Lemma~\ref{lem:tree_hom}) gives $t_{B_{2k}}(W)\geq t_{D_k}(W)^2/t_{C_4}(W)$,
and thus,
    \begin{align*}
        m_{B_{2k}}(W) \geq \frac{m_{D_k}(W)^2}{m_{C_4}(W)} \geq \frac{m_D(W)^{2k}}{m_{K_{1,2}}(W)^{2k-2}m_{C_4}(W)}
    \end{align*}
by Lemma~\ref{lem:tm}. Then again by \eqref{eq:goodmanD} and Lemma~\ref{lem:diamond},
\begin{align*}
    m_{B_{2k}}(W) \geq \frac{c\cdot m_{D}(W)^{2k}}{\left(\frac{2}{3}\sqrt{m_{D}(W)}+\frac{1}{3}\right)^{2k-2} \left(m_D(W)-\frac{1}{16}+\frac{c}{8}\right)}.
\end{align*}
It remains to minimise rational function 
\begin{align*}
    h_{k,c}(x):= \frac{16\cdot 3^{2k-2}cx^{4k}}{(2x+1)^{2k-2}(16x^2-1+2c)}
\end{align*}
of $x:=\sqrt{m_D(W)}$ subject to $x\geq 1/4$. 
Taking $c=1/7$, $h_{k,1/7}$ is a positive constant times the function $g_k$, where the derivative
\begin{align*}
    g_k'(x)= (2x+1)^{1-2k}(112x^2-5)^{-2}x^{4k-1}(112kx^3+(112k-56)x^2-(5k+5)x-5k).
\end{align*}
Thus, it suffices to check $p_k(x)=112kx^3+(112k-56)x^2-(5k+5)x-5k>0$ on $x\geq 1/4$. Rearranging the terms we get $p_k(x) = 112k(x-1/4)^3+(196k-56)(x-1/4)^2+(72k-33)(x-1/4)+(10k-19)/4$, which is positive for $x\geq 1/4$ and $k\geq 2$.
Therefore, $h_{k,1/7}(x)$ is minimised when $x=1/4$, which implies $B_{2k}$ is common. 
\end{proof}

We remark that the constant $c=1/7$ has been judiciously chosen. Indeed, if $c$ is too large, then it gets tougher or even impossible to obtain the inequality in Lemma~\ref{lem:diamond}. Otherwise, if $c$ is too small, then the rational function $h_c(x)$ may attain its local minimum at some $x_0>1/4$ and the optimisation does not work. This does happen if one tries to apply the same argument to prove that $K_{2,2,t}$ is common for $t>2$.

However, flag algebras allow us to prove inequalities that resemble Lemma~\ref{lem:diamond} and can be directly applied to~\eqref{eq:222v122}, which gives tighter bounds than the previous approach does. In particular, the following generalises Lemma~\ref{lem:diamond} to any connected bipartite graph on at most $5$ vertices. 

\def\Bipapexcommscale{0.8}
\begin{figure}
{\hskip 1em
    \begin{tikzpicture}[scale=\Bipapexcommscale]
    \draw
    \foreach \x in {0,1}{
    (60+180*\x:0.9) node[vtx](x\x){}
    }
    (x0)--(x1)
    ;
    \draw (0,-1.6)node{$K_{2}$};
    \end{tikzpicture} 
\hfill
    \begin{tikzpicture}[scale=\Bipapexcommscale]
    \draw
    \foreach \x in {0,1,2}{
    (30+120*\x:1) node[vtx](x\x){}
    }
    (x0)--(x2)--(x1)
    ;
    \draw (0,-1.8)node{$K_{1,2} = P_3$};
    \end{tikzpicture}   
\hfill
    \begin{tikzpicture}[scale=\Bipapexcommscale]
    \draw
    \foreach \x in {0,1,2,3}{
    (45+90*\x:1) node[vtx](x\x){}
    }
    (x2)--(x0) (x2)--(x1)  (x2)--(x3)
    ;
    \draw (0,-1.5)node{$K_{1,3}$};
    \end{tikzpicture}    
\hfill
    \begin{tikzpicture}[scale=\Bipapexcommscale]
    \draw
    \foreach \x in {0,1,2,3}{
    (45+90*\x:1) node[vtx](x\x){}
    }
    (x3)--(x0) (x2)--(x1)  (x2)--(x3)
    ;
    \draw (0,-1.4)node{$P_4$};
    \end{tikzpicture}
\hfill
    \begin{tikzpicture}[scale=\Bipapexcommscale]
    \draw
    \foreach \x in {0,1,2,3}{
    (45+90*\x:1) node[vtx](x\x){}
    }
    (x3)--(x0)--(x1) (x2)--(x1)  (x2)--(x3)
    ;
    \draw (0,-1.5)node{$C_4$};
    \end{tikzpicture}
\hskip 1em}
\vskip 0.5em
{\hskip 1em
    \begin{tikzpicture}[scale=\Bipapexcommscale]
    \draw
    \foreach \x in {0,1,2,3,4}{
    (90+72*\x:1) node[vtx](x\x){}
    }
    (x0)--(x1) (x0)--(x2) (x0)--(x3) (x0)--(x4)
    ;
    \draw (0,-1.5)node{$K_{1,4}$};    
    \end{tikzpicture}
\hfill
    \begin{tikzpicture}[scale=\Bipapexcommscale]
    \draw
    \foreach \x in {0,1,2,3,4}{
    (90+72*\x:1) node[vtx](x\x){}
    }
    (x0)--(x1) (x0)--(x2) (x0)--(x4) (x3)--(x4)
    ;
    \draw (0,-1.5)node{$P_4*_u^v K_2$};    
    \end{tikzpicture}   
\hfill
    \begin{tikzpicture}[scale=\Bipapexcommscale]
    \draw
    \foreach \x in {0,1,2,3,4}{
    (90+72*\x:1) node[vtx](x\x){}
    }
    (x0)--(x1)--(x2)  (x0)--(x4)--(x3)
    ;
    \draw (0,-1.5)node{$P_5$};    
    \end{tikzpicture}   
\hfill
    \begin{tikzpicture}[scale=\Bipapexcommscale]
    \draw
    \foreach \x in {0,1,2,3,4}{
    (90+72*\x:1) node[vtx](x\x){}
    }
    (x3)--(x0)--(x2) (x1)--(x3)--(x4)--(x2)
    ;
    \draw (0,-1.5)node{$K_{2,3}^-$};
    \end{tikzpicture}    
\hfill
    \begin{tikzpicture}[scale=\Bipapexcommscale]
    \draw
    \foreach \x in {0,1,2,3,4}{
    (90+72*\x:1) node[vtx](x\x){}
    }
    (x3)--(x0)--(x2)--(x1)--(x3)--(x4)--(x2)
    ;
    \draw (0,-1.5)node{$K_{2,3}$};    
    \end{tikzpicture}
\hskip 1em}
    \caption{The ten connected bipartite graphs on at most 5 vertices from Lemma~\ref{lem:flag}.}
    \label{fig:my_label}
\end{figure}
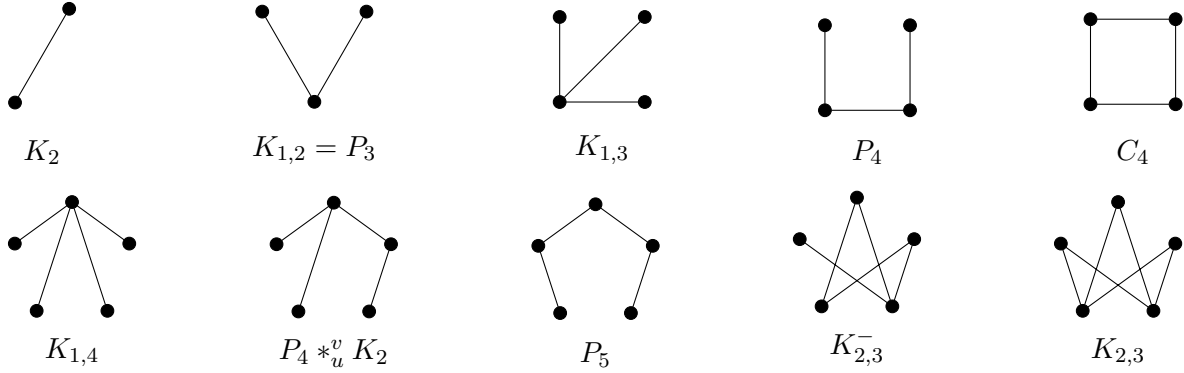

\begin{lemma}\label{lem:flag}
If $H$ is a connected bipartite graph on at most $5$ vertices and $W$ is a graphon, then
\begin{align}
m_{H^{+1}}(W) \geq 2^{-v(H)} \cdot m_{H}(W).
\label{eq:c}
\end{align}
Moreover, if $H\neq K_2$, then $m_{H^{+1}}(W) = 2^{-v(H)} \cdot m_{H}(W)$ if and only if $m_{C_4}(W)=1/8$.
\end{lemma}
\begin{proof}
For any of the ten considered graphs, the proof of \eqref{eq:c} is a straightforward flag algebra application. As the proof of \eqref{eq:c} for $H\neq K_2$ uses $m_{C_4}(W) \ge 1/8$, the moreover part follows by complementary slackness.
The flag algebra calculations certifying \eqref{eq:c} can be downloaded from \url{http://lidicky.name/pub/common/}.
\end{proof}

Given a common graph $H$, if the inequality \eqref{eq:c} holds then a direct application of Lemma~\ref{lem:tm} yields that $H^{+a}$ is common for every positive integer $a$. In particular, we are now ready to prove Theorem~\ref{thm:Krst}.

\begin{proof}[Proof of Theorem~\ref{thm:Krst}]
Fix an integer $a\ge2$ and a graph $H$. By convexity (or Lemma~\ref{lem:tree_hom}), we have
\begin{align*}
    t_{H^{+a}}(W) \geq \frac{t_{H^{+1}}(W)^a}{t_{H}(W)^{a-1}}\,.
\end{align*}
Lemma~\ref{lem:tm} then yields that
\begin{align*}
    m_{H^{+a}}(W) \geq \frac{m_{H^{+1}}(W)^a}{m_{H}(W)^{a-1}}\,,
\end{align*}
and thus, by Lemma~\ref{lem:flag}, we conclude that
\begin{align*}
    m_{H^{+a}}(W) \geq \frac{m_{H}(W)}{2^{a \cdot v(H)}}\,.
\end{align*}
Since $H$ is common, i.e.,  $m_{H}(W) \geq 2^{1-e(H)}$, we have that $m_{H^{+a}}(W) \geq 2^{1-e(H)-a\cdot v(H)} = 2^{1-e(H^{+a})}$.
In other words, the graph $H^{+a}$ is common.
\end{proof}

\section{Concluding remarks}

\textbf{Stability.} When a graph $H$ is known to be common, it is natural to ask a stability question, i.e., whether the random colouring is (asymptotically) the unique minimiser of the number of monochromatic copies of $H$. In other words, is $m_H(W)$ uniquely minimised by $W=1/2$ almost everywhere?
For bipartite graphs, this question connects to the so-called \emph{Forcing Conjecture}~\cite{SkokanThoma04,CFS10} stating that if $H$ is bipartite with at least one cycle and $p \in (0,1)$, then $W=p$ almost everywhere uniquely minimises the number of copies of $H$ among all graphons of density $p$.

For our results, one may check that that the random colouring is the unique minimiser of $m_H$ whenever $H$ is a triangle-tree with $\kappa(H)\ge 1$, i.e., a triangle-tree that is not a triangle-vertex-tree. Indeed, as both $W$ and $1-W$ must be tight 
for \eqref{eq:Jtree_hom}, inspecting the proof of~\cite[Theorem~2.7]{L19} yields that any minimiser of $m_H$ must be $1/2$-regular and have the `correct' codegrees, i.e.,
$\int W(x,y) \mathrm{d}y=1/2$ and $\int W(x,z)W(z,y)\mathrm{d}z=1/4$ for almost every $x,y\in[0,1]$, respectively. 
In particular, Lemma~\ref{lem:diamond} and its applications immediately proves that $K_{1,1,2}$, $K_{1,2,2}$, and $K_{2,2,2}$ has a unique minimiser. 
On the other hand, there are infinitely many minimisers of $m_{K_3}$. Indeed, $m_{K_3}(W)=1/4$ for every $1/2$-regular graphon $W$.
Analogusly, any $1/2$-regular graphon minimizes $m_H$ when $H$ is a fixed triangle-vertex-tree. 

In all the cases covered in Theorem~\ref{thm:Krst} except $H=K_2$, the `moreover' part of Lemma~\ref{lem:flag} yields that the random colouring is the unique minimiser. 
When $H=K_2$, the graph $H^{+a}$ is simply the complete tripartite graph $K_{1,1,a}$. Therefore, the case $a=1$ corresponds to $H^{+a}=K_3$, so every $1/2$-regular graphon minimises $m_{K_3}$. On the other hand, if $a \ge 2$, then $H^{+a}$ is a triangle-tree with $\kappa(H^{+a})=a-1$, hence by the discussion in the previous paragraph, $m_{H^{+a}}(W)$ is uniquely minimised when $W=1/2$ almost everywhere.

\vspace{3mm}
\medskip

\noindent
\textbf{Theorem~\ref{thm:tritree} for odd cycles.} It is certainly possible to generalise Theorem~\ref{thm:tritree} by replacing triangles by odd cycles. One way is to define \emph{$C_{2k+1}$-vertex-tree} and \emph{$C_{2k+1}$-edge-tree} by allowing recursive additions of odd cycles along vertices or edges, respectively. 
It is then straightforward to check these graphs are common by using~\ref{lem:tree_hom}. 
Furthermore, one may also generalise Theorem~\ref{thm:tritree} by allowing both types of vertex- and edge-additions of odd cycles of length $2k+1$; 
however, it is unclear that one can allow even more general additive operation between odd cycles, e.g., along a multi-edge path.
It might be interesting to obtain a full generalisation of Theorem~\ref{thm:tritree} along this line to obtain that $C_{2k+1}$-trees are common for every $k$.

\vspace{2mm}
\medskip\noindent
\textbf{Optimal pendant trees.} Let $H$ be a common graph. Then one may ask what is the smallest $T$ that makes $T*_{u}^{v}H$ uncommon. To formalise, let
\begin{align*}
    \mathrm{UC}(H) :=\min\{e(T): T*_{u}^{v}H \text{ is uncommon}\}.
\end{align*}
Note that this parameter might not exist for some bipartite graphs $H$. Indeed, if $H$ satisfies Sidorenko's conjecture, then $T*_{u}^{v}H$ satisfies the conjecture as well. In particular, $H$ is common, and we let $\mathrm{UC}(H)=\infty$.
On the other hand, if $H$ is a triangle-edge-tree, then Lemma~\ref{lem:add_tree} and the proof of Theorem~\ref{thm:adding_tree} yield a lower bound for $\mathrm{UC}(H)$ that is linear in $e(H)$.
Also, Fox's result~\cite{F08} implies that $\mathrm{UC}(K_{t,t,t})=O(t^2)$, which is again linear in terms of the number of edges. It would be interesting to see more precise estimates for $\mathrm{UC}(H)$ for various non-bipartite graphs $H$.

\vspace{2mm}
\medskip\noindent
\textbf{Ramsey multiplicity constant of small graphs.}
The smallest graph whose Ramsey multiplicity constant is not known is $K_4$, and determining the value of $C(K_4)$ is a well-known open problem in extremal combinatorics
with no conjectured value. A direct flag algebra calculation using expressions with 9-vertex subgraph densities yields $C(K_4) \geq 1/33.77 \approx  0.0296$, which is a slight improvement over previously known lower bound
$1/33.9739 \approx 0.0294343$
\cite{Gir79,Sperf12,Niess12,MR624553,MR2675922,Vaughan}. Unfortunately, there is still a non-negligible gap from the best upper bound $1/33.0205 \approx 0.030284$ by Even-Zohar and Linial~\cite{MR3386015}, who improved an upper bound $1/33.0135$ of Thomason~\cite{Thom89,Thom97}.

As noted in the introduction, flag algebra method can be used to prove commonality of many small graphs. In Appendix we give a proof that the graphs $H_3$ and $H_4$ from Wolf's list on Figure~\ref{fig:smallcommon} are common. 
Although it is possible to fully inspect the presented proof by hand, some of the steps were obtained by using computers. It would still be interesting to find simpler `human-friendly' proofs of the commonality of $H_3$ or $H_4$.

\vspace{5mm}

\noindent\textbf{Acknowledgements.} Part of this work was carried out when the first, second, and third authors met in Seoul for One-day Meeting on
 Extremal Combinatorics. We would like to thank the organisers of the workshop for their hospitality. The second author is grateful to David Conlon for helpful discussions and comments.
We are also grateful to anonymous referees and Steve Butler for helping us to improve the presentation of this paper.

\bibliographystyle{abbrvurl}
\bibliography{references}

\pagebreak
\appendix\section{Proof of commonality of $H_3$ and $H_4$}
\setcounter{MaxMatrixCols}{16}
We present proofs of the inequalities $m_{H_3}(W) \ge 2^{-5}$ and $m_{H_4}(W) \ge 2^{-6}$ for all graphons $W$, where $H_3$ and $H_4$ are depicted on Figure~\ref{fig:smallcommon}.
The proofs were obtained with a computer assistance using libraries CSDP~\cite{CSDP} and QSOPT~\cite{QSOPT}.

Firstly, the following three subgraph density expressions will evaluate to a nonnegative number for every graphon $W$ due to the commonality of the corresponding graphs:
\[
           \mbox{(1) } 465\cdot\left(m_{H_1}(W) - 2^{-5}\right) \,,
\quad\quad \mbox{(2) } 465\cdot\left(m_{H_2}(W) - 2^{-5}\right) \, \quad\mbox{ and }
\quad\quad \mbox{(3) } 48\cdot\left(m_{C_5}(W) - 2^{-4}\right) \,.
\]
Moreover, each of these expression will be written as a linear combination of $5$-vertex \emph{induced} subgraph densities; recall that $\tau_H(W)$, the induced density of~$H$ in~$W$, is defined as follows:
$$
     \tau_H(W):=\E
     \left[
    \prod_{ij\in E(H)}W(x_i,x_j)\prod_{ij\notin E(H)} (1-W(x_i,x_j))
    \right].
$$
As we aim to exploit the symmetry of the colours in Ramsey multiplicity, we let $f(\tau_H(W)) := \tau_H(W) + \tau_{\overline{H}}(W)$ for every graph $H$ and extend $f$ linearly to formal linear combinations of graphs.

Let $P_{ab}(W)$ be the probability measure on $[0,1]^2$ which, given a graphon $W$, corresponds to a uniformly sampled pair $(a,b)$ that induces an edge.
Let $T_{\emptyset}(W)$ and $T_{bc}(W)$ be the probability measures on $[0,1]^3$ that correspond to sampling $(a,b,c)$ inducing an independent set and a single-edge graph $\{bc\}$, respectively. We consider the following 13 density expressions represented as sum-of-squares (we note that (6) and (7) were suggested by a computer search): %
\def\f{{\mbox{\large $f$}}}
{\small\begin{align*}
\mbox{\large (4)} \quad & 10 \cdot \f\left(\E_{T_{\emptyset}(W)}\left[\left(
  \PP\left[x \in \textstyle\bigcap_{z\in\{a,b,c\}} N_z \right]
- \PP\left[x \notin \textstyle\bigcup_{z\in\{a,b,c\}} N_z\right]
\right)^2\right] \right) \,,
\\
\mbox{\large (5)} \quad & 10 \cdot  \f\left(\E_{T_{\emptyset}(W)}\left[\left(
  8\cdot\PP\left[x \notin \textstyle\bigcup_{z\in\{a,b,c\}} N_z\right] - 1
\right)^2\right] \right) \,,
\\
\mbox{\large (6)} \quad & 30 \cdot \f\left(\E_{T_{bc}(W)}\left[\Big(
  2\cdot \PP[x \in N_b \Delta N_c]
+ 3\cdot \PP\left[x \in (N_b\Delta N_c) \setminus N_a \right]
\Big)^2\right] \right) \,,
\\
\mbox{\large (7)} \quad & 30 \cdot \f\left(\E_{T_{bc}(W)}\left[ \Big(
  2 \cdot \PP[x \in N_b \Delta N_c]
- 7 \cdot \PP\left[x \in (N_b\Delta N_c) \setminus N_a \right]
\Big)^2\right] \right) \,,
\\
\mbox{\large (8)} \quad & 30 \cdot \f\left(\E_{T_{bc}(W)}\left[ \left(
  \PP\left[x \in \textstyle\bigcap_{z\in\{a,b,c\}} N_z\right]
- \PP\left[x \notin \textstyle\bigcup_{z\in\{a,b,c\}} N_z\right]
\right)^2 \right] \right) \,,
\\
\mbox{\large (9)} \quad & 30 \cdot \f\left(\E_{T_{bc}(W)}\left[ \left(
  \PP\left[(x \in N_a \setminus \textstyle\bigcup_{z\in\{b,c\}} N_z\right]
- \PP\left[x \notin \textstyle\bigcup_{z\in\{a,b,c\}} N_z\right]
\right)^2 \right] \right) \,,
\\
\mbox{\large(10)} \quad & 30 \cdot \f\left(\E_{T_{bc}(W)}\left[ \left(
  \PP\left[x \in \textstyle\bigcap_{z\in\{b,c\}} N_z \setminus N_a\right]
- \PP\left[x \notin \textstyle\bigcup_{z\in\{a,b,c\}} N_z\right]
\right)^2 \right] \right) \,,
\\
\mbox{\large (11)} \quad & 30 \cdot \f\left(\E_{T_{bc}(W)}\left[ \left(
  \PP\left[x \in (N_b\Delta N_c) \setminus N_a\right]
- 2\cdot \PP\left[x \notin \textstyle\bigcup_{z\in\{a,b,c\}} N_z\right]
\right)^2 \right] \right) \,,
\\
\mbox{\large (12)} \quad & 30 \cdot \f\left(\E_{T_{bc}(W)}\left[ \left(
\PP\left[x \in (N_b\Delta N_c) \cap N_a\right]
- 2\cdot \PP\left[x \notin \textstyle\bigcup_{z\in\{a,b,c\}} N_z\right]
\right)^2 \right] \right) \,,
\\
\mbox{\large (13)} \quad & 15 \cdot \E_{a\in[0,1]}\left[ \Big(
\big(2\cdot\PP[x\in N_y]-1\big) \cdot 
\big( \PP\left[x\in N_a \land y \in N_a \right] - \PP\left[x\notin N_a \land y \notin N_a \right] \big)
\Big)^2 \right] \,,
\\
\mbox{\large (14)} \quad & 15 \cdot \f\left(\E_{P_{ab}(W)}\left[ \left(
\PP\left[x \in N_a\right] - \PP\left[x \in N_b\right]
\right)^2 \right] \right) \,,
\\
\mbox{\large (15)} \quad & 15 \cdot \E_{a\in[0,1]}\left[ \left(
2\cdot\PP\left[x\in N_a \land y \in N_a \right] + 2\cdot\PP\left[x\notin N_a \land y \notin N_a \right] - 1
\right)^2 \right],\text{ and }
\\
\mbox{\large (16)} \quad & 30 \cdot \f\left(\E_{P_{ab}(W)}\left[ 
\PP[x \in N_a \cap N_b] \cdot \big(2\cdot \PP\left[y \in N_a \Delta N_b\right] - 1\big)^2 \right] \right) \,,
\end{align*}}%
where $x$ and $y$ are uniformly sampled vertices of $W$,
and $x \in N_{\star}$ abbreviates the event of sampling an edge between $x$ and $\star$.
Clearly, each expression evaluates to a nonnegative number for any $W$ and can be written as a linear combination of $5$-vertex induced subgraph densities.

\begin{figure}\def\FIVEgraphsscale{0.515}
\begin{tikzpicture}[scale=\FIVEgraphsscale]
\draw\foreach \i in {0,1,2,3,4}{(90+72*\i:1) node[vtx](v\i){}};
\end{tikzpicture}
\hfill
\begin{tikzpicture}[scale=\FIVEgraphsscale]
\draw\foreach \i in {0,1,2,3,4}{(90+72*\i:1) node[vtx](v\i){}};
\draw(v3)--(v2);
\end{tikzpicture}
\hfill
\begin{tikzpicture}[scale=\FIVEgraphsscale]
\draw\foreach \i in {0,1,2,3,4}{(90+72*\i:1) node[vtx](v\i){}};
\draw(v2)--(v3)--(v4);
\end{tikzpicture}
\hfill
\begin{tikzpicture}[scale=\FIVEgraphsscale]
\draw\foreach \i in {0,1,2,3,4}{(90+72*\i:1) node[vtx](v\i){}};
\draw(v1)--(v2)(v3)--(v4);
\end{tikzpicture}
\hfill
\begin{tikzpicture}[scale=\FIVEgraphsscale]
\draw\foreach \i in {0,1,2,3,4}{(90+72*\i:1) node[vtx](v\i){}};
\draw(v1)--(v4)(v2)--(v4)(v3)--(v4);
\end{tikzpicture}
\hfill
\begin{tikzpicture}[scale=\FIVEgraphsscale]
\draw\foreach \i in {0,1,2,3,4}{(90+72*\i:1) node[vtx](v\i){}};
\draw(v2)--(v3)(v2)--(v4)(v3)--(v4);
\end{tikzpicture}
\hfill
\begin{tikzpicture}[scale=\FIVEgraphsscale]
\draw\foreach \i in {0,1,2,3,4}{(90+72*\i:1) node[vtx](v\i){}};
\draw(v2)--(v3) (v1)--(v0)--(v4);
\end{tikzpicture}
\hfill
\begin{tikzpicture}[scale=\FIVEgraphsscale]
\draw\foreach \i in {0,1,2,3,4}{(90+72*\i:1) node[vtx](v\i){}};
\draw(v1)--(v2)--(v3)--(v4);
\end{tikzpicture}
\hfill
\begin{tikzpicture}[scale=\FIVEgraphsscale]
\draw\foreach \i in {0,1,2,3,4}{(90+72*\i:1) node[vtx](v\i){}};
\draw(v4)--(v0)--(v1) (v2)--(v0)--(v3);
\end{tikzpicture}
\vskip 1.5em
\begin{tikzpicture}[scale=\FIVEgraphsscale]
\draw\foreach \i in {0,1,2,3,4}{(90+72*\i:1) node[vtx](v\i){}};
\draw(v4)--(v1)--(v0)--(v4)--(v3);
\end{tikzpicture}
\hfill
\begin{tikzpicture}[scale=\FIVEgraphsscale]
\draw\foreach \i in {0,1,2,3,4}{(90+72*\i:1) node[vtx](v\i){}};
\draw(v2)--(v3)(v0)--(v4)(v1)--(v4)(v3)--(v4);
\end{tikzpicture}
\hfill
\begin{tikzpicture}[scale=\FIVEgraphsscale]
\draw\foreach \i in {0,1,2,3,4}{(90+72*\i:1) node[vtx](v\i){}};
\draw(v3)--(v4)--(v0)--(v1)--(v2);
\end{tikzpicture}
\hfill
\begin{tikzpicture}[scale=\FIVEgraphsscale]
\draw\foreach \i in {0,1,2,3,4}{(90+72*\i:1) node[vtx](v\i){}};
\draw(v1)--(v2)--(v3)--(v4)--(v1);
\end{tikzpicture}
\hfill
\begin{tikzpicture}[scale=\FIVEgraphsscale]
\draw\foreach \i in {0,1,2,3,4}{(90+72*\i:1) node[vtx](v\i){}};
\draw(v0)--(v1)--(v4)--(v0) (v2)--(v3);
\end{tikzpicture}
\hfill
\begin{tikzpicture}[scale=\FIVEgraphsscale]
\draw\foreach \i in {0,1,2,3,4}{(90+72*\i:1) node[vtx](v\i){}};
\draw(v2)--(v3)(v0)--(v4)(v1)--(v4)(v2)--(v4)(v3)--(v4);
\end{tikzpicture}
\hfill
\begin{tikzpicture}[scale=\FIVEgraphsscale]
\draw\foreach \i in {0,1,2,3,4}{(90+72*\i:1) node[vtx](v\i){}};
\draw(v4)--(v3)--(v2)--(v1)--(v4)--(v0);
\end{tikzpicture}
\hfill
\begin{tikzpicture}[scale=\FIVEgraphsscale]
\draw\foreach \i in {0,1,2,3,4}{(90+72*\i:1) node[vtx](v\i){}};
\draw(v0)--(v1)--(v4)--(v0) (v1)--(v2) (v4)--(v3);
\end{tikzpicture}
\hfill
\begin{tikzpicture}[scale=\FIVEgraphsscale]
\draw\foreach \i in {0,1,2,3,4}{(90+72*\i:1) node[vtx](v\i){}};
\draw(v0)--(v1)--(v2)--(v3)--(v4)--(v0);
\end{tikzpicture}
\caption{Non-isomorphic partitions of $E(K_5)$ into two parts represented by black and white edges.}\label{fig:all5graphs}
\end{figure}
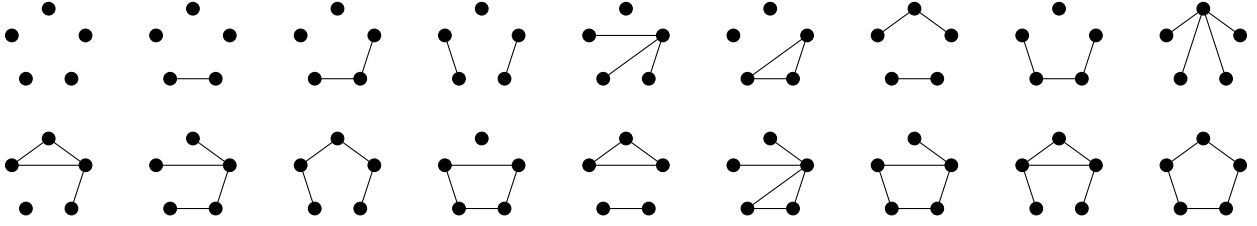

As there are $34$ non-isomorphic $5$-vertex graphs and two of them are self-complementary,
there are exactly $18$ non-isomorphic partitions of $E(K_5)$ into two parts (see Figure~\ref{fig:all5graphs}).
Therefore, we may identify each expression described in the previous paragraph with a vector from $\RR^{18}$ simply by letting its $i$-th coordinate to be the coefficient of the $i$-th graph in Figure~\ref{fig:all5graphs} in the corresponding linear combination. We denote these vectors by $w_1, w_2, \dots, w_{16}$, and let $M := \left(w_1 | w_2 | \cdots | w_{16}\right) $ be the corresponding $18 \times 16$ matrix. Next, let $v_A$ and $v_B$ be the vectors from $\RR^{18}$ representing the expressions $480\cdot\left(m(H_3) - 2^{-5}\right)$ and $960\cdot\left(m(H_4) - 2^{-6}\right)$, respectively. 
Then $v_A$,~$v_B$, and $M$ are
\begin{equation*}{\footnotesize \setlength{\arraycolsep}{3pt}
\renewcommand{\arraystretch}{0.85}
\begin{pmatrix}
465\\177\\33\\81\\-15\\-15\\17\\1\\-15\\-15\\-15\\-7\\-15\\-15\\-15\\-15\\-15\\-15
\end{pmatrix},
\begin{pmatrix}
945\\273\\17\\113\\-15\\-15\\17\\-15\\-15\\-15\\-15\\-15\\-15\\-15\\-15\\-15\\-15\\-15
\end{pmatrix}, \text{ and }
\begin{pmatrix}
 465& 465& 45& 10& 490& 0& 0& 0& 0& 0& 0& 0& 15& 0& 15& 0\\
 177& 177& 21& 1& 7& 0& 0& 3& 3& 3& 12& 12& 3& 0& 3& 3\\
 49& 65& 5& 0& -27& 0& 0& 1& 0& 0& -8& 0& -3& 1& 3& -3\\
 49& 17& 13& 0& 4& 0& 0& 2& -2& 2& 8& 8& 3& -2& -9& 6\\
 9& 9& -3& -1& -28& 75& 12& 0& 0& 0& 3& 0& -3& 3& -9& 6\\
 -15& 33& -3& 1& 4& 0& 0& 0& 0& -6& 0& 0& -3& 3& 15& 6\\
 1& -15& 5& 0& 2& 0& 0& 1& -1& 0& 0& 0& 1& -3& -9& 0\\
 1& 1& 1& 0& 3& -25& -4& 0& 0& 0& -3& -4& -1& 0& 3& -5\\
 -15& -15& -3& -4& -28& 0& 0& 0& 0& 6& 0& 0& 3& 6& -9& 6\\
 -15& -7& -3& 0& 3& 25& 4& -1& 0& -1& 1& 0& 1& 3& 3& -1\\
 -7& -15& -3& 0& 2& 24& 5& 0& 0& 0& 0& -3& -1& 0& -9& 0\\
 -15& -15& 1& 0& 1& -24& -5& 0& 0& 0& 0& 1& 1& -3& 3& -3\\
 -15& -15& -3& 0& 2& -100& -16& 0& 0& 0& 4& 0& -1& 0& 15& 8\\
 -15& -15& -3& 1& 1& 0& 0& 0& 3& 0& 0& 0& -3& -3& 3& 9\\
 -15& -15& -3& 0& 3& 0& 0& -2& 0& 1& 0& 0& 3& 4& 3& -1\\
 -15& -15& -3& 0& 1& -16& 45& 0& 0& 0& 0& 1& -1& -2& 3& 1\\
 -15& -15& -3& 0& 2& 40& -40& 0& 0& 0& 0& 0& 1& 1& -9& 0\\
 -15& -15& 5& 0& 0& -40& -250& 0& 0& 0& 0& 10& 5& -5& 15& 10\\
\end{pmatrix},
}
\end{equation*}
respectively.
Let $M_A$ and $M_B$ be the submatrices of $M$ obtained by deleting the last and the second to last column, respectively.
It follows both $M_A$ and $M_B$ have rank $15$ and the unique $x_A$ and~$x_B$ that satisfy $v_A = M_A x_A$ and $v_B = M_B x_B$ have nonnegative entries, explicitly given as follows:

{\footnotesize\renewcommand{\arraystretch}{0.85}
\begin{align*}
&
x_A = \frac{1}{133168} \times \begin{pmatrix}
 22852\\
 10730\\
 448079\\
 6584\\
 13776\\
 1168352\\
 22852\\
 1351168\\
 9280\\
 513184\\
 43384\\
 7888\\
 172057\\
 329614\\
 45472
\end{pmatrix}
&
x_B = \frac{1}{13601} \times \begin{pmatrix}
 9628\\
 2465\\
 19430\\
 780\\
 2520\\
 56144\\
 9628\\
 133168\\
 19952\\
 19488\\
 14268\\
 6728\\
 71746\\
 14268\\
 18676
\end{pmatrix}
\;.
\end{align*}}%
Thus, $m_{H_3}(W)\ge 2^{-5}$ and $m_{H_4}(W) \ge 2^{-6}$ for every graphon~$W$.
\end{document}